\documentclass[12pt]{article}

\input epsf
\usepackage{amsmath}
\usepackage{amssymb}
\usepackage{amsthm}
\usepackage{pstricks-add}

\usepackage{graphicx}
\usepackage[arrow, matrix, curve]{xy}

\usepackage[margin=3cm]{geometry}
\usepackage{array, xcolor}
	\definecolor{lightgray}{gray}{0.8}
	\newcolumntype{L}{>{\raggedleft}p{0.14\textwidth}}
	\newcolumntype{R}{p{0.8\textwidth}}
	
\usepackage[ansinew]{inputenc}

\makeatletter
\newtheorem*{rep@theorem}{\rep@title}
\newcommand{\newreptheorem}[2]{%
\newenvironment{rep#1}[1]{%
 \def\rep@title{#2 \ref{##1}}%
 \begin{rep@theorem}}%
 {\end{rep@theorem}}}
\makeatother

\newtheorem{theorem}{Theorem}
\newreptheorem{theorem}{Theorem}
\newtheorem{claim}{Claim}
\newtheorem{lemma}{Lemma}
\newtheorem{proposition}{Proposition}

\theoremstyle{remark}\newtheorem{remark}{Remark}
\theoremstyle{definition}\newtheorem{definition}{Definition}

\newcommand{\R}{\mathbb R}

\newcommand{\Z}{\mathbb Z}

\newcommand{\id}{\operatorname{id}}
\newcommand{\Stab}{\operatorname{Stab}}

\def\ga{\gamma}
\def\Ga{\Gamma}

\def\De{\Delta}

\def\Si{\Sigma}

\def\3{\ss}

\def\acts{\curvearrowright}
\def\D{\partial}
\def\geo{\partial_{\infty}}
\def\interior{\operatorname{int}}

\def\pihalf{\frac{\pi}{2}}

\def\2pithird{\frac{2\pi}{3}}

\def\tangle{\angle_{Tits}}

\begin{document}

\title{An obstruction to the smoothability 
of singular nonpositively curved metrics 
on 4-manifolds by patterns of incompressible tori}
\author{Stephan Stadler \thanks{E-mail address: \texttt{sstadler@math.uni-koeln.de}}}
%\date{December 7, 2013}

\maketitle

\begin{abstract}
We give new examples of closed smooth 4-manifolds 
which support singular metrics of nonpositive curvature, 
but no smooth ones, 
thereby answering affirmatively a question of Gromov. 
The obstruction comes from patterns of incompressible 2-tori 
sufficiently complicated to force branching of geodesics
for nonpositively curved metrics. 
\end{abstract}

\tableofcontents

\section{Introduction}

The goal of this note is to exhibit new examples of closed smooth 4-manifolds 
which support singular metrics of nonpositive (sectional) curvature, 
but no smooth ones. 
Such manifolds had first been found by 
Davis, Januszkiewicz and Lafont \cite{DJL}. 
The approaches are different, 
but they both rely on a basic rigidity phenomenon in nonpositive curvature, 
namely that free abelian subgroups in the fundamental group 
of a closed nonpositively curved manifold 
are carried by totally-geodesically immersed flat tori. 
The fundamental groups $\Ga$ 
of the singular locally CAT(0) 4-manifolds $M$ 
studied in \cite{DJL} 
contain few (``isolated'') copies of $\Z^2$. 
The obstruction to the existence 
of a smooth nonpositively curved metric on $M$ is 
that they, respectively, 
the corresponding invariant flats in the universal covering $\widetilde M$, 
are {\em knotted at infinity}. 
This is impossible for 2-flats in smooth Hadamard 4-manifolds. 
We consider fundamental groups $\Ga$ 
which contain {\em plenty} of copies of $\Z^2$ 
and exploit the fact that this rigidifies the geometry 
of nonpositively curved metrics on $M$, singular or smooth, 
since it enforces a complicated pattern of 
%totally-geodesically 
immersed flat 2-tori. 
Extreme cases occur in ``higher rank'': 
When $\Ga$ splits as a product $\Ga_1\times\Ga_2$ of subgroups, 
then the universal cover splits as a metric product. 
Or (in dimensions $\geq5$) 
when $\Ga$ is the fundamental group of an
irreducible higher rank locally symmetric space 
of noncompact type, 
then the geometry of nonpositively curved metrics is completely rigidified 
by Mostow rigidity, 
i.e.\ it is essentially unique up to rescaling.
We consider ``rank one'' situations 
where there are still plenty of subgroups isomorphic to $\Z^2$
which however only partially rigidify the geometry. 
Heuristically, 
singular nonpositively curved metrics 
allow more complicated patterns of tori than smooth ones 
because, due to possible branching, 
the tori can be packed ``more densely''. 
It is therefore conceivable that sufficiently complicated patterns 
which occur for singular nonpositively curved metrics
cannot occur in the smooth case
because they enforce the branching of geodesics. 
Indeed, natural candidates 
to which this line of reasoning could apply 
had been pointed out by Gromov 
in (the first exercise of) \cite{BGS}, 
namely the fundamental groups of branched coverings 
$M\to\Si\times\Si$ 
of products of higher genus surfaces $\Si$ with themselves 
with branching locus the diagonal. 
The purpose of this note is to do this exercise.
More precisely, we prove

\begin{theorem}[Exercise 1 in \cite{BGS}]\label{mainthm}
Let $V$ be a closed 4-dimensional manifold which admits a non-trivial finite branched covering $\beta:V\rightarrow\Sigma\times\Sigma$ 
over the product of a hyperbolic surface $\Sigma$ with itself 
such that the 
branching locus equals the diagonal $\Delta_\Sigma\subset\Sigma\times\Sigma$. 
Then $V$ admits no smooth Riemannian metric
of nonpositive sectional curvature. 
\end{theorem}

The above theorem is an application of a more general result (Theorem \ref{thm:obst}), which provides
an obstruction for a discrete group $\Ga$ to act geometrically on 
a Hadamard manifold. The obstruction comes from the existence 
of a CAT(0) model space $X_{model}$ which contains a
specific ``singular configuration'' and
admits a geometric action\\ $\Ga\acts X_{model}$. 

In short, the singular configuration in $X_{model}$ consists of two rigid convex subsets.
Their rigidity ensures that one finds corresponding convex subsets in every CAT(0) space which allows 
for a geometric action by $\Ga$. The singular nature of $\Ga$ is reflected in
the way these two sets interact. On the one hand, their correlation forces 
the presence of branching geodesics. On the other hand, they are inseparable, in 
the sense that their interaction persists when passing to corresponding sets in 
other CAT(0) spaces with geometric actions by $\Ga$. Consequently, {\em any} CAT(0)
space which permits a geometric action by $\Ga$ has to be singular.

\textbf{Acknowledgments}. The results in this paper are part of 
my upcoming thesis \cite{thesis}.
I am very grateful to my advisor, Bernhard Leeb, for 
calling my attention to the first exercise in \cite{BGS} and
for his support and guidance during the last years. I would 
also like to thank Alexander Lytchak and Richard Bamler for 
many related and unrelated discussions
as well as the geometry group at LMU 
for providing an inspiring working environment.

\section{Preliminaries}

For notations and basics on CAT(0) spaces 
we refer the reader to the 
first two chapters of \cite{B} and Section 2
of \cite{KL}.

\subsection{Quasi-isometry invariance of flats}

A {\em flat} $F$ in a CAT(0) space $X$ is a convex subset 
isometric to a Euclidean space. If a flat has dimension
$k$, then we will also call it a {\em k-flat}.
%A flat is called \emph{maximal} 
%if it is not contained in a strictly larger flat. 

If $\Ga\acts X$ is an isometric action, 
then a flat $F\subset X$ is called \emph{$\Gamma$-periodic} 
if its stabilizer $\Stab_\Gamma(F)$ acts cocompactly on it. 
In case of a discrete action, 
a finite index subgroup of $\Stab_\Gamma(F)$ 
acts on $F$ by translations, 
and hence $\Stab_\Gamma(F)$ 
is virtually free abelian
of rank equal to the dimension of $F$. 

An isometric action $\Ga\acts X$ 
of a discrete group $\Ga$ on a locally compact CAT(0) space $X$ 
is called {\em geometric} 
if it is properly discontinuous and cocompact.
Then every abelian subgroup $A\subset\Ga$ 
preserves a flat in $X$ on which it acts cocompactly. 

Suppose that $\Ga\acts X$ and $\Ga\acts X'$ are geometric actions 
of the same group on two locally compact CAT(0) spaces. 
Then there exists a $\Ga$-equivariant quasi-isometry $\Phi:X\rightarrow X'$. 

The following properties of CAT(0) spaces will imply Proposition \ref{prop:uniformHcl} 
which will be used further in the proof.

If $F\subset X$ is a $\Ga$-periodic flat, 
then its stabilizer $\Stab_\Gamma(F)$ is virtually abelian 
and preserves a flat $F'\subset X'$. 
Hence, 
a $\Ga$-equivariant quasi-isometry $\Phi:X\rightarrow X'$ 
carries $\Ga$-periodic flats in $X$ 
Hausdorff close to $\Ga$-periodic flats in $X'$. 

One can also say something 
regarding the quasi-isometry invariance of {\em non}-periodic flats: 
Recall that, 
for locally compact CAT(0) spaces with cocompact isometry group, 
the maximal dimension of flats 
equals the maximal dimension of quasi-flats
and is in particular a quasi-isometry invariant 
\cite[Thm.\ C]{K}. 
By Theorem B in \cite{LS}, 
quasi-flats of maximal dimension which are within finite Hausdorff 
distance from (maximal) flats are actually within uniformly bounded 
Hausdorff distance from these flats.
Combining these results, 
one obtains the following useful 
\begin{proposition}\label{prop:uniformHcl}
There exists a constant $D=D(L,A,X,X')$ such that
a $\Ga$-equivariant $(L,A)$-quasi-isometry $\Phi:X\to X'$, 
between CAT(0) spaces $X$ and $X'$,
maps $\Gamma$-periodic flats of maximal dimension in $X$ 
$D$-Hausdorff close to such flats in $X'$. 
As a consequence, 
also pointed Hausdorff limits 
of $\Gamma$-periodic flats of maximal dimension in $X$ 
are carried $D$-Hausdorff close to such flats in $X'$.
\end{proposition}

\subsection{Product rigidity}

Recall that an isometry of a CAT(0) space is called {\em axial},
if it preserves a complete geodesic on which it acts as a nontrivial translation.
Such a geodesic is then called an {\em axis}.
We will need the following product splitting result which is a special case of Corollary 10 in \cite{M}.
See also Proposition 2.2 in \cite{L} and Theorem 1 in \cite{S}.

\begin{proposition}\label{prop:split}
Let $X$ be a locally compact CAT(0) space 
and let $\Gamma\cong\Gamma_1\times\Gamma_2$ be a product of non-abelian 
free groups $\Gamma_i$. Suppose that $\Gamma$ acts on $X$ discretely 
by axial isometries. Then there exists a minimal non-empty  
$\Gamma$-invariant closed convex subset $C\subset X$ which splits 
metrically as a product, $C\cong C_1\times C_2$, such that $\Gamma$ preserves the product 
splitting and $\Gamma_i$ acts trivially on $C_{3-i}$.
\end{proposition}

\begin{remark}
If $X$ is 4-dimensional, then the set $C$ is unique. Indeed,
the factors $C_i$ would have to be 2-dimensional. Since
two minimal non-empty $\Gamma$-invariant closed convex 
subsets are parallel and $X$ is 4-dimensional, $C$ is unique.
\end{remark}

\subsection{Coarse intersection of flats and quasi-isometry invariance}

For a subset $A$ of a metric space $X$ we denote 
its closure by $\overline{A}$ and its tubular $r$-neighborhood
by $N_r(A)$.

Let $F_1,F_2\subset X$ be flats. 
We say that they {\em diverge} 
if $\partial_{\infty} F_1\cap\partial_{\infty} F_2=\emptyset$. 
Equivalently, 
the distance function $d(\cdot,F_2)|_{F_1}$ is proper 
and grows (at least) linearly. 

\begin{definition}\label{def:sta}
Let $F_1,F_2\subset X$ be diverging flats. 
We say that $F_1$ {\em coarsely intersects} $F_2$ 
if there exists $R\geq 0$ such that for every $r\geq R$ holds: 
If $B_1\subset F_1$ is a round ball such that 
$F_1\cap\overline{ N_r(F_2)}\subset\interior(B_1)$, 
then its boundary sphere $\partial B_1$ is not contractible 
inside $X\setminus\overline{ N_r(F_2)}$. 
\end{definition}

\begin{remark}
(i) 
This is independent of the choice of the ball $B_1\subset F_1$.

(ii) 
The notion is asymptotic 
in the sense that it only depends on the ideal boundaries of the 
flats, i.e.\ passing to parallel flats does not affect coarse intersection.

(iii) 
Coarse intersection is not a symmetric relation.

(iv) 
In general, disjoint flats can coarsely intersect. 
However, this cannot occur in geodesically complete smooth spaces, i.e.\ in Hadamard manifolds. 
\end{remark}

We need a criterion to recognize whether flats coarsely intersect. In the 
smooth case ``coarse intersection'' simply becomes
``nontrivial transversal intersection'', i.e.\
two flats in a Hadamard manifold intersect coarsely 
if and only if they intersect transversely in one point. 
This is clear, because for a flat $F$ in a Hadamard manifold $X$ 
there is a deformation retraction of $X\setminus F$ onto $X\setminus \overline{N_r(F)}$ 
using the gradient flow of $d(\cdot,F)$. 

More generally, we have:
\begin{lemma}\label{lem:intcrit}
Let $F_1$ and $F_2$ be flats in a CAT(0) space $X$. 
Suppose that 
$F_2$ is contained in an open convex subset $C\subset X$ which is Riemannian, 
i.e.\ the metric on $C$ is induced by a smooth Riemannian metric. 
If
$F_1$ and $F_2$ intersect transversely in one point, then 
$F_1$ coarsely intersects $F_2$. 
\end{lemma}
\begin{proof}
Otherwise 
spheres in $F_1\setminus F_2$ around the intersection point $F_1\cap F_2$ 
could be contracted in $X\setminus F_2$. 
But this would be absurd 
since $X\setminus F_2$ retracts to $\overline C\setminus F_2$ along normal geodesics. 
\end{proof}

It will be crucial for us
that coarse intersection is quasi-isometry invariant.

\begin{lemma}\label{lem:corinv}
Let $\Phi:X\to X'$ be a quasi-isometry  of CAT(0) spaces
with a quasi-inverse $\Phi':X'\to X$. 
Let $F_1,F_2\subset X$ and $F'_1,F'_2\subset X'$ be flats 
such that $\Phi(F_i)$ is Hausdorff close to $F'_i$. 
Then $F_1$ coarsely intersects $F_2$ if and only if  
$F'_1$ coarsely intersects $F'_2$.
\end{lemma} 
\begin{proof}
First note that $F_1$ and $F_2$ diverge if and only if $F'_1$ and $F'_2$ do.
The quasi-flat $\Phi|_{F_1}$ may not be continuous, 
but since $X$ is CAT(0), 
it is uniformly (in terms of the quasi-isometry constants) 
Hausdorff close to a continuous quasi-flat $q:F_1\to X'$.
Suppose that $q(F_1)$ is $D$-Hausdorff close to $F'_1$.
If $F'_1$ does not coarsely intersect $F'_2$, then for 
every tubular neighborhood $N_r(F'_2)$ of $F'_2$ all
large spheres in $F'_1$ are contractible in the complement
of $\overline{N_r(F'_2)}$. Because the $q$-image of a sphere in 
$F_1$ can be homotoped to $F'_1$ by a $D$-short homotopy, we obtain that 
$q$-images of large spheres in $F_1$ are contractible in 
$X\setminus \overline{N_r(F'_2)}$.
The $\Phi'$-image of a (contracting) homotopy is again 
uniformly Hausdorff close to a continuous map. 
Since 
$\Phi'\circ\Phi$ is at finite distance from $\id_X$, 
it follows that we can for every radius $r>0$ 
contract sufficiently large spheres in $F_1$ 
in the complement of the tubular $r$-neighborhood of $F_2$. 
Consequently, 
$F_1$ does not coarsely intersect $F_2$.
\end{proof}

\section{Configurations of convex product subsets in dimension 4}

\subsection{Flat half-strips in CAT(0) surfaces with symmetries}

By a {\em flat strip}, respectively, {\em half-strip} 
of {\em width} $w\geq0$ in a CAT(0) space 
we mean a convex subset isometric to $\R\times[0,w]$, respectively, 
to $[0,+\infty)\times[0,w]$. 

The following observation restricts the possible positions 
of flat half-strips in a CAT(0) surface 
relative to the action of its isometry group. 
\begin{lemma}\label{lem:dim2}
Let $Y$ be a smooth CAT(0) surface, 
and let $h\subset Y$ be a flat half-strip. 
Suppose that $h$ is asymptotic to a periodic geodesic $c\subset Y$, 
i.e.\ to an axis $c$ of an axial isometry $\ga$ of $Y$. 

Then either $w=0$, 
or $h$ extends to a (periodic) flat strip in $Y$ parallel to $c$. 
\end{lemma}
\begin{proof}
We may assume that $\ga$ translates towards the ideal endpoint of $h$ 
and preserves the orientation transversal to $c$. 
If $w>0$ and $r(t)$ is a ray in $\interior(h)$, 
then the ray $\gamma^{-1}r$ is strongly asymptotic to $r$, 
i.e.\ $d(\gamma^{-1}r(t),r)\to0$ as $t\to+\infty$. 
Therefore, $\gamma^{-1}r$ must enter $\interior(h)$,
because $\interior(h)$ is open in $Y$. 
Consequently, 
$\gamma^{-1}r$ extends $r$, 
and $\gamma^{-1}h$ extends $h$.
It follows by induction that $h$ is contained in a $\ga$-invariant 
flat strip. 
\end{proof}

\subsection{Configurations not occuring in smooth spaces}
\label{sec:confnotoccsm}

Let $X$ be a CAT(0) space. 

We describe a configuration of convex product subsets
which can occur if $X$ is singular, 
but not if it is smooth.

We assume that $X$ contains two closed convex subsets, 
namely a product 
\begin{equation*}
Y_1\times Y_2
\end{equation*}
of smooth CAT(0) surfaces $Y_1$ and $Y_2$ with boundary 
such that $\interior(Y_1)\times\interior(Y_2)$ is open in $X$; 
and a product 
\begin{equation*}
Z\times\R
\end{equation*}
whose (not necessarily smooth) cross section $Z$ 
contains an ideal triangle 
with three ideal vertices $\eta_0,\eta_+,\eta_-$. 
We denote the sides asymptotic to $\eta_0$ and $\eta_{\pm}$ by $l_\mp$ 
and the side asymptotic to $\eta_+$ and $\eta_-$ by $l_0$.

\begin{figure}[ht]
\begin{center}
\psset{xunit=1.0cm,yunit=1.0cm,algebraic=true,dimen=middle,dotstyle=o,dotsize=3pt 0,linewidth=0.8pt,arrowsize=3pt 2,arrowinset=0.25}
\scalebox{0.7}{
\begin{pspicture*}(-3.204192352216151,-3.3543901994838143)(8.668142142557365,6.44718827875942)
\parametricplot{2.919612924257734}{3.9668104754543316}{1.*8.16272013485701*cos(t)+0.*8.16272013485701*sin(t)+11.8|0.*8.16272013485701*cos(t)+1.*8.16272013485701*sin(t)+3.66}
\parametricplot{5.014008026650929}{6.0612055778475264}{1.*8.162720134857008*cos(t)+0.*8.162720134857008*sin(t)+-4.124871130596426|0.*8.162720134857008*cos(t)+1.*8.162720134857008*sin(t)+7.254228634059951}
\parametricplot{0.8252178218645385}{1.8724153730611361}{1.*8.16272013485701*cos(t)+0.*8.16272013485701*sin(t)+0.7248711305964259|0.*8.16272013485701*cos(t)+1.*8.16272013485701*sin(t)+-8.334228634059949}
\rput[tl](3.997436905272435,5.8){$\eta_0$}
\rput[tl](-2.3,-0.6013485767087883){$\eta_+$}
\rput[tl](6.612428126945448,-2.3269785904840056){$\eta_-$}
\rput[tl](1.3,2.3286846458825683){$l_{-}$}
\rput[tl](4.388579708394818,1.70745784092349){$l_{+}$}
\rput[tl](2.2487984913135453,-0.6473653770761274){$l_{0}$}
\rput[tl](5.331924115925272,3.341054253964029){\large$Z$}
\end{pspicture*}}
\caption{The ideal triangle contained in $Z$.}\label{fig1}
\end{center}
\end{figure}
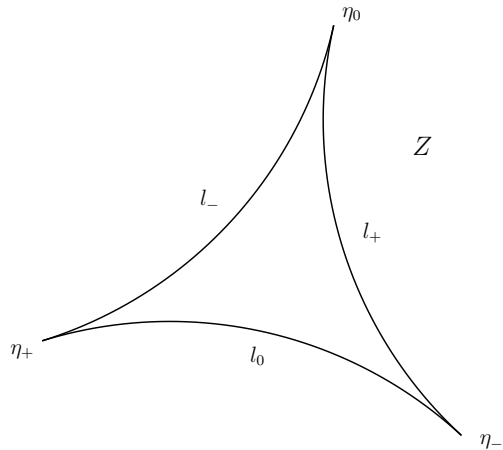

We assume furthermore, 
that these product subsets 
interact as follows \footnote{See Figure \ref{fig2}.}: 

{\bf(i)} 
The intersection of the flat $F_{\pm}=l_{\pm}\times\R\subset Z\times\R$ 
with $Y_1\times Y_2$ contains a quadrant $r^{\pm}_1\times r^{\pm}_2$, 
where $r^{\pm}_i$ are asymptotic rays in $Y_i$. 
We denote their common ideal endpoint by $\xi_i\in\geo Y_i$. 

{\bf(ii)} 
$\eta_0$ is an {\em interior} point of the Tits arc $\xi_1\xi_2$ 
of length $\pihalf$ in $\geo X$. 

\begin{figure}[ht]
\begin{center}
\psset{xunit=1.0cm,yunit=1.0cm,algebraic=true,dimen=middle,dotstyle=o,dotsize=3pt 0,linewidth=0.8pt,arrowsize=3pt 2,arrowinset=0.25}
\scalebox{0.7}{
\begin{pspicture*}(-2.9090526229496714,-5.623670432318329)(13.979826430103431,7.399634123578131)
\pscircle(4.64,0.32){5.084761547998098}
\psline(6.52,2.76)(2.76,-2.76)
\psline(2.76,2.76)(2.76,-2.76)
\psline(6.52,2.76)(6.52,-2.76)
\psline(4.64,2.76)(4.64,0.14267134192816533)
\psline[linewidth=1.8pt](4.64,2.76)(4.64,0.1567504337735337)
\psline[linewidth=1.8pt](4.64,-2.76)(4.64,-0.1567504337735337)
\psline(4.64,5.404761547998098)(4.64,2.76)
\psline(4.64,-2.76)(4.64,-4.764761547998098)
\psline(2.76,2.76)(4.498289364447486,0.20804327347071228)
\psline(6.52,-2.76)(4.781710635552513,-0.20804327347071228)
\psline(4.64,2.76)(5.236524726698814,1.0085018662885878)
\psline(5.3032016153789945,0.8127271718659304)(6.52,-2.76)
\psline(6.52,2.76)(5.629186259983677,0.14441923144143676)
\psline(5.547717938602044,-0.0947856270833598)(5.302064301120079,-0.816066520115512)
\psline(5.205215978879483,-1.1004296790347086)(4.64,-2.76)
\psline(2.76,2.76)(3.9861671941362715,-0.8402355912937334)
\psline(4.049213268023108,-1.0253495954721021)(4.64,-2.76)
\psline(4.64,2.76)(4.057324537974843,1.049165664692092)
\psline(3.9744425394807537,0.8058100095392358)(3.7483129989580712,0.1418551884300836)
\psline(3.655593454070926,-0.13038517740876943)(2.76,-2.76)
\rput[tl](4.9,3){$\xi_1$}
\rput[tl](4.9,-2.5){$\xi_2$}
\rput[tl](4.9,0.3){$\eta_0$}
\rput[tl](9.9,0.3){$\eta_+$}
\rput[tl](-0.2,0.3){$\eta_-$}
\rput[tl](5.9,3.2){$\geo Y_1$}
\rput[tl](5.9,-2.9){$\geo Y_2$}
\rput[tl](9,4){$\geo (Z \times \R)$}
\begin{scriptsize}
\psdots[dotstyle=*,linecolor=black](4.64,2.76)
\psdots[dotstyle=*,linecolor=black](4.64,-2.76)
\psdots[dotstyle=*,linecolor=black](-0.4346822560629363,0.)
\psdots[dotstyle=*,linecolor=black](9.714682256062934,0.)
\psdots[dotstyle=*,linecolor=black](4.64,0.1816624613696635)
\end{scriptsize}
\end{pspicture*}}
\caption{The configuration in $\geo X$.}\label{fig2}
\end{center}
\end{figure}

Then the intersection $Y_1\times Y_2 \cap Z\times\R$ is nonempty and, 
by condition (ii), 
the product structures (i.e.\ the directions of the factors) 
do not match on it. 
The latter implies 
that the convex subset $Y_1\times Y_2 \cap Z\times\R$ 
is {\em flat}.\footnote{This follows from the fact 
that a geodesic triangle in a product of CAT(0) spaces 
is flat if and only if its projections to the factors are flat.} 
As a consequence, 
subrays of the rays $r^{\pm}_i$ bound a {\em flat half-strip} 
$h_i\subset Y_i$.

In addition, 
we impose a {\em symmetry} condition: 

{\bf(iii)} The rays $r^{\pm}_i$ are asymptotic to a periodic geodesic 
$c_i\subset Y_i$. 

Using Lemma \ref{lem:dim2} above, we conclude:
Either subrays of the rays $r^{\pm}_i$ coincide, 
or subrays extend to geodesics $c^{\pm}_i\subset Y_i$ parallel to $c_i$. 

\begin{claim}\label{cl1}
If conditions (i)-(iii) hold, 
then $X$ cannot be smooth.
\end{claim}

\begin{proof}
Suppose that $X$ is smooth.
Then our discussion implies 
that the flats $F_{\pm}$ 
either have a quadrant in common and therefore coincide, 
or contain parallel half-planes and their intersection of ideal boundaries 
$\geo F_+\cap\geo F_-$ contains an arc of length $\pi$ 
of the form $\xi_1\xi_2\hat\xi_1$ or $\xi_2\xi_1\hat\xi_2$ 
with an antipode $\hat\xi_i\in\geo Y_i$ for $i=1$ or $2$. 
It follows that $\tangle(\eta_{\pm},\hat\xi_i)<\pihalf$ 
and hence $\tangle(\eta_+,\eta_-)<\pi$, a
contradiction. 
\end{proof}
 
\subsection{Not equivariantly smoothable configurations}
\label{sec:nequivsm}

Now, we restrict to {\em symmetric} situations 
and consider {\em geometric} actions
\begin{equation*}
\Ga\acts X
\end{equation*}
by discrete groups on locally compact CAT(0) spaces,
i.e.\ actions which are isometric, properly discontinuous and cocompact.

We will tie the configuration considered above 
sufficiently closely to the action
so that it will carry over to other geometric actions $\Ga\acts X'$ 
on CAT(0) spaces. 
This will then be used to rule out such actions 
on {\em smooth} CAT(0) spaces, 
i.e.\ on Hadamard 4-manifolds. 

In addition to the conditions (i)-(iii) above, we assume:

{\bf(iv)}
$X$ contains no 3-flats. 

{\bf(v)}
$Y_1\times Y_2$ is preserved by a subgroup $\Ga_1\times\Ga_2\subset\Ga$ 
with non-abelian free factors $\Ga_i$, 
and the restricted action $\Ga_1\times\Ga_2\acts Y_1\times Y_2$ 
is a product action 
(not necessarily cocompact). 

{\bf(vi)}
The flats $F_{\pm}$ and the flat $F_0=l_0\times\R$ in $Z\times\R$
are {\em $\Ga$-periodically approximable} 
(i.e.\ pointed Hausdorff limits of $\Ga$-periodic flats). 

{\bf(vii)}
The geodesics $c_i\subset Y_i$ are $\Ga_i$-periodic. 
Moreover, there exist $\Ga_i$-periodic geodesics $d_i\subset\interior(Y_i)$ 
which intersect the rays $r^{\pm}_i\subset Y_i$ transversally in points. 

\medskip
Under the assumptions (i)-(vii),
we look for a corresponding configuration in $X'$.
Let $\Phi:X\to X'$ denote a $\Ga$-equivariant quasi-isometry. 

By (v) and Proposition~\ref{prop:split}, 
there exists a $\Ga_1\times\Ga_2$-invariant closed convex product subset 
(in general singular)
\begin{equation*}
Y'_1\times Y'_2\subset X'
\end{equation*}
on which $\Ga_1\times\Ga_2$ acts by a product action. 
The $\Ga_i$-periodic image quasigeodesics $\Phi(c_i)$ are Hausdorff close 
to $\Ga_i$-periodic geodesics $c'_i\subset Y'_i$. 

By (iv+vi) and Proposition~\ref{prop:uniformHcl}, 
the quasi-flats $\Phi(F_{\pm})$ and $\Phi(F_0)$ 
are Hausdorff close to flats $F'_{\pm}$ and $F'_0$. 
We have that 
any one of these flats is contained in a tubular neighborhood 
of the union of the other two. 
Hence its ideal boudary circle 
is contained in the union of the ideal boudary circles of the other two. 
This leaves only the possibility that the union of their ideal boudaries 
is a spherical suspension of three points. 
It follows that the flats are contained in a closed convex product subset
\begin{equation*}
Z'\times\R\subset X'
\end{equation*}
whose cross section $Z'$ contains an ideal triangle 
with corresponding ideal vertices $\eta'_0,\eta'_+,\eta'_-$ 
and sides $l'_+,l'_-,l'_0$, 
such that 
$F'_{\pm}=l'_{\pm}\times\R$ 
and $F'_0=l'_0\times\R$. 
Furthermore, 
if $\rho\subset Z$ is a ray asymptotic to one of the ideal vertices 
$\eta_0,\eta_+$ or $\eta_-$, 
then $\Phi$ carries 
the vertical half-plane $\rho\times\R\subset Z\times\R$
Hausdorff close to a vertical half-plane $\rho'\times\R\subset Z'\times\R$
where $\rho'\subset Z'$ is a ray with corresponding ideal endpoint 
$\eta'_0,\eta'_+$ or $\eta'_-$. 
This follows from the fact that the half-plane $\rho\times\R$ 
is Hausdorff close to the intersection of sufficiently large 
tubular neighborhoods of two of the flats 
$F_{\pm}$ and $F_0$. 

Since the $c_i$ are periodic, 
$\Phi$ carries the quadrants $r^{\pm}_1\times r^{\pm}_2$
Hausdorff close to a quadrant $r'_1\times r'_2$ 
for rays $r'_i\subset c'_i$. 
The quadrants $r^{\pm}_1\times r^{\pm}_2$ 
are contained in vertical half-planes 
with ideal boundary semicircle $\geo F_+\cap\geo F_-$ 
and, by condition (ii), 
their ideal boundary arc $\xi_1\xi_2$ of length $\pihalf$
is contained in the interior of this semicircle.
Denoting the ideal endpoints of the rays $r'_i$ by $\xi'_i=\geo r'_i$ 
it follows that the arc $\xi'_1\xi'_2$ of length $\pihalf$ 
is contained in the {\em interior} of the 
semicircle $\geo F'_+\cap\geo F'_-$, 
and $\eta'_0$ is an {\em interior} point 
of the arc $\xi'_1\xi'_2$ of length $\pihalf$.

In summary, 
the interaction of the product subsets $Y'_1\times Y'_2$ and $Z'\times\R$ 
at infinity is as for the configuration in $X$. 
However, 
without further assumptions, 
the intersection
$Y'_1\times Y'_2 \cap Z'\times\R$ 
could be empty. 

\begin{claim}
$X'$ cannot be smooth Riemannian. 
\end{claim}
 
\begin{proof}
Suppose that $X'$ is smooth. 
We show that then the intersection 
$Y'_1\times Y'_2 \cap Z'\times\R$ 
must be nonempty. 

Note that there exist $\Ga_i$-periodic geodesics $d'_i\subset Y'_i$ 
with the same stabilizers as the geodesics $d_i$. 
By (vii), 
the periodic flat $d_1\times d_2$ 
transversally intersects the flats $F_{\pm}$ in points 
inside the smooth region 
$\interior(Y_1)\times\interior(Y_2)$. 
Hence, by Lemma \ref{lem:intcrit}, 
$d_1\times d_2$ {\em coarsely} intersects $F_{\pm}$. 
It follows from Lemma \ref{lem:corinv} that $d'_1\times d'_2$ 
coarsely intersects $F'_{\pm}$.
Now we use that $X'$ is  smooth to deduce that 
$d'_1\times d'_2$ intersects $F'_{\pm}$ transversally in a point. 
In particular, 
$Y'_1\times Y'_2 \cap Z'\times\R\neq\emptyset$. 

It follows that conditions (i)-(iii) are satisfied by 
the product subsets $Y'_1\times Y'_2$ and $Z'\times\R$ of $X'$. 
By Claim \ref{cl1}, 
this is a contradiction.
\end{proof}

We have proved:
\begin{theorem}[Obstruction to smooth action]
\label{thm:obst}
If a discrete group $\Ga$ admits a geometric action $\Ga\acts X$ 
on a locally compact CAT(0) space 
satisfying conditions (i)-(vii), 
then $\Ga$ does not act geometrically on any smooth Hadamard 4-manifold.
\end{theorem}

\begin{remark}
The regularity assumptions can be relaxed. 
The argument works more generally and shows that
$\Ga$ does not act geometrically on locally compact, 
geodesically complete CAT(0) spaces $X'$ without branching geodesics, 
for instance $\mathcal{C}^2$-smooth Hadamard 4-manifolds \cite{thesis}. 
\end{remark}

\section{An example}

In this section,
we consider the geometric actions on 
4-dimensional singular CAT(0) spaces 
suggested by Gromov in the first exercise of \cite{BGS}
and verify that they contain configurations 
satisfying conditions (i-vii). 

%\subsection{The geometry of certain branched coverings}

Let $\Si$ be a 
%connected orientable 
closed surface of genus $\geq2$, 
and let 
\begin{equation*}
\beta:V\to\Si\times\Si
\end{equation*}
be a non-trivial finite branched covering 
with branching locus the diagonal $\De_{\Si}\subset\Si\times\Si$. 
Then the group 
\begin{equation*}
\Ga:=\pi_1(V) 
\end{equation*}
admits geometric actions on 4-dimensional singular CAT(0) spaces: 
Let $\pi_V:X\to V$ denote the universal covering, 
and $\pi:=\beta\circ\pi_V:X\to\Si\times\Si$. 
We equip $\Si$ with a hyperbolic metric and pull back 
the corresponding product metric on $\Si\times\Si$ 
to singular metrics on $V$ and $X$. 
In this way the 4-manifold $X$ becomes a CAT(0) space, 
and the deck action 
\begin{equation*}
\Ga\acts X
\end{equation*}
becomes a geometric action. 

Regarding the geometry of $X$, 
note first that 
the {\em singular locus} $\pi^{-1}(\De_{\Si})\subset X$
is a disjoint union of isometrically embedded hyperbolic planes. 
The restriction of $\pi$ to any of them is a universal covering of 
the {\em branching locus} $\De_{\Si}\subset\Si\times\Si$. 

We look for patterns of flats in $X$ which obstruct the existence of 
geometric $\Ga$-actions on Hadamard manifolds, 
as described in sections~\ref{sec:confnotoccsm} and~\ref{sec:nequivsm}. 

The space $X$ contains no 3-dimensional flats, 
but plenty of 2-dimensional ones. 
There are two kinds of them: 
flats disjoint from $\pi^{-1}(\De_{\Si})$, 
and flats which intersect $\pi^{-1}(\De_{\Si})$ 
orthogonally in one or several parallel geodesics. 

Let ${\mathcal F}_0$ denote 
the set of flats disjoint from $\pi^{-1}(\De_{\Si})$. 
There are obvious subfamilies of ${\mathcal F}_0$ 
which occur in convex product subsets of $X$.
Namely, let 
\begin{equation}
\label{eq:deco}
\Si=\Si^+\cup\Si^- 
\end{equation}
be a decomposition of $\Si$ 
into two subsurfaces $\Si^{\pm}$ 
along a finite family of disjoint closed geodesics.  
Then the open product block $\interior(\Si^+\times\Si^-)\subset\Si\times\Si$ 
is disjoint from $\De_{\Si}$, 
and hence the connected components of its inverse image 
$\pi^{-1}(\interior(\Si^+\times\Si^-))$ 
in $X$ are convex subsets isometric to 
$\interior(\widetilde\Si^+\times\widetilde\Si^-)$ 
on which $\pi$ restricts to a universal covering of 
$\interior(\Si^+\times\Si^-)$. 

The other flats in $X$ important for our argument are,
somewhat unexpectedly, 
the flats which intersect $\pi^{-1}(\De_{\Si})$ 
in precisely {\em one} geodesic; 
let us denote the set of these flats by ${\mathcal F}_1$. 
Understanding them 
leads us to considering flat half-planes. 

We define ${\mathcal H}$ as the set of 
{\em injectively} immersed flat half-planes 
$H\subset\Si\times\Si$ 
which intersect the branching locus precisely along their boundary line, 
$H\cap\De_{\Si}=\D H$,  
and are orthogonal to it, $H\perp\De_{\Si}$. 
Furthermore, 
we define $\widetilde{\mathcal H}$ as the set of 
isometrically embedded flat half-planes
$\widetilde H\subset X$ 
such that 
$\widetilde H\cap\pi^{-1}(\De_{\Si})=\D\widetilde H$
and $\widetilde H\perp\pi^{-1}(\De_{\Si})$.
We say that a half-plane $\widetilde H\in\widetilde{\mathcal H}$ 
{\em covers} or is a {\em lift} of a half-plane $H\in{\mathcal H}$ 
if $\pi|_{\widetilde H}$ is a local isometry onto $H$. 
A flat in ${\mathcal F}_1$ is the union of two 
half-planes in $\widetilde{\mathcal H}$ with common boundary line. 

We collect some facts about 
${\mathcal H}$ and $\widetilde{\mathcal H}$ 
needed for our argument. 

If $H\in {\mathcal H}$, 
then $\D H$ is an injectively immersed line in $\De_{\Si}$ 
and therefore of the form 
$\D H=\De_c$ for a nonperiodic simple geodesic $c\subset\Si$. 
It follows that $H\subset c\times c$ because $H$ is flat. 
We also see that half-planes in ${\mathcal H}$ occur in pairs 
of opposite half-planes with common boundary line. 

A half-plane $H\in {\mathcal H}$ lifts to 
a half-plane $\widetilde H\in\widetilde{\mathcal H}$ 
because it is simply-connected 
and the branched covering $\beta$ is a true covering over 
$\Si\times\Si-\De_{\Si}$. 
More precisely, 
for a point $p\in H-\D H$ 
and a lift $\widetilde p$ of $p$ 
there exists a unique lift $\widetilde H$ of $H$ 
with $\widetilde p\in\widetilde H$. 
A lift $\widetilde l\subset\pi^{-1}(\De_{\Si})$ of the boundary line $\D H$ 
extends in several ways 
to a lift $\widetilde H$ of $H$, 
because points close to $\D H$ can be lifted in several ways to points 
close to $\widetilde l$. 
The number of lifts is given by the local branching order
of $\pi$ at $\widetilde l$. 

If $\widetilde H\in\widetilde{\mathcal H}$, 
then its boundary line $\D\widetilde H$ 
projects to an immersed line $\De_c$ in $\De_{\Si}$. 
The geodesic $c\subset\Si$ must be nonperiodic simple, 
because otherwise 
$(\widetilde H-\D\widetilde H)\cap\pi^{-1}(\De_{\Si})\neq\emptyset$. 
Thus, all half-planes in $\widetilde{\mathcal H}$ are lifts 
of half-planes in ${\mathcal H}$. 

If $\widetilde H_1,\widetilde H_2\in\widetilde{\mathcal H}$ 
are distinct half-planes with the same boundary line, 
$\D\widetilde H_1=\D\widetilde H_2$, 
then their projections $H_1,H_2\in{\mathcal H}$ 
either coincide or are a pair of opposite half-planes. 
The local geometry of branched coverings implies, 
that $\widetilde H_1,\widetilde H_2$ have angle $\pi$ 
along their common boundary line 
and their union 
$\widetilde H_1\cup\widetilde H_2$
is a flat in ${\mathcal F}_1$.

We will use the following consequence of this discussion: 
Let $c\times c\subset\Si\times\Si$ be an injectively immersed plane, 
and let $H_{\pm}$ be the half-planes into which it is divided by $\De_c$. 
Then for every lift $\widetilde H_+$ of $H_+$ 
there exist at least two distinct lifts 
$\widetilde H_-^1,\widetilde H_-^2$ of $H_-$ 
with the same boundary line $\D\widetilde H_-^i=\D\widetilde H_+$, 
and the union of any two of the three half-planes 
$\widetilde H_+,\widetilde H_-^1,\widetilde H_-^2$ 
is a flat in ${\mathcal F}_1$. 

\medskip
The flats in ${\mathcal F}_1$ are nonperiodic. 
Nevertheless, 
they are useful for investigating 
geometric $\Ga$-actions on other CAT(0) spaces. 
This is due to the following fact: 

\begin{lemma}\label{lem:limit}
Let $F\in{\mathcal F}_1$. 
Suppose that the nonperiodic simple geodesic 
$\pi(F\cap\pi^{-1}(\De_{\Si}))$ in $\De_{\Si}$ 
is the pointed Hausdorff limit of periodic simple geodesics in $\De_{\Si}$. 
Then $F$ is the pointed Hausdorff limit of $\Ga$-periodic flats in $X$. 
\end{lemma}

\begin{proof}
We denote $\widetilde l=F\cap\pi^{-1}(\De_{\Si})$. 
Let $(c_n,p_n)\to (c,p)$ be a sequence of 
pointed periodic simple geodesics in $\Si$ 
converging to the nonperiodic simple geodesic $c\subset\Si$ 
with $\pi(\widetilde l)=\De_c$. 
There exist geodesics 
$\widetilde l_n\subset\pi^{-1}(\De_{\Si})$ 
lifting the $c_n$ 
and lifts 
$\widetilde p_n,\widetilde p$ of the base points $p_n,p$ 
such that $(\widetilde l_n,\widetilde p_n)\to (\widetilde l,\widetilde p)$. 
We choose embedded subsegments $s_n\subset c_n$ 
of increasing lengths centered at the base points $p_n$ 
such that also $(s_n,p_n)\to (c,p)$ 
and lifted segments $\widetilde s_n\subset\widetilde l_n$
centered at the $\widetilde p_n$ 
such that $(\widetilde s_n,\widetilde p_n)\to (\widetilde l,\widetilde p)$. 

The main step of the argument is to approximate $F$ 
by isometrically embedded flat squares 
$\widetilde Q_n\subset\pi^{-1}(s_n\times s_n)$ 
with diagonals $\widetilde s_n$, 
$(\widetilde Q_n,\widetilde p_n)\to(F,\widetilde p)$. 
This will imply the assertion 
because isometrically embedded flat squares in 
$\pi^{-1}(c_n\times c_n)$ are contained in $\Ga$-periodic flats. 
Indeed, 
the subsets $\pi^{-1}(c_n\times c_n)\subset X$ 
have cocompact stabilizers in $\Ga$,
and their connected components are convex subsets 
which split as metric products 
of the line with discrete metric trees.
All flats contained in them are limits of $\Ga$-periodic ones. 

To find the squares $\widetilde Q^n$, 
we proceed as follows. 
The flat $F$ is divided by $\widetilde l$ into two half-planes 
$\widetilde H_{\pm}\in\widetilde{\mathcal H}$. 
We will approximate these simultaneously by 
isometrically embedded right-angled isosceles triangles 
$\widetilde T_{\pm}^n\subset\pi^{-1}(s_n\times s_n)$
with sides $\widetilde s_n$. 

Let $\widetilde q_{\pm}\in\widetilde H_{\pm}-\D\widetilde H_{\pm}$
be base points close to $\widetilde p$, 
and let 
$\bar q_{\pm}=\pi(\widetilde q_{\pm})\in c\times c-\De_c$ 
denote their projections. 
There exist sequences of points 
$\bar q_{\pm}^n\in s_n\times s_n-\De_{s_n}$ approximating them, 
$\bar q_{\pm}^n\to \bar q_{\pm}$. 
More precisely, 
we choose them such that they are close to $\De_{p_n}\in\De_{s_n}$ 
intrinsically in $s_n\times s_n$, 
i.e.\ such that the segments $\De_{p_n}\bar q_{\pm}^n\subset s_n\times s_n$. 
Furthermore, 
there exists a sequence of lifts 
$\widetilde q_{\pm}^n\in\pi^{-1}(\bar q_{\pm}^n)$
close to $\widetilde p_n$
such that $\widetilde q_{\pm}^n\to\widetilde q_{\pm}$. 

The injectively immersed square $s_n\times s_n\subset\Si\times\Si$
is divided by $\De_{s_n}$ into two triangles. 
Let $T_{\pm}^n$ be the subtriangle containing $\bar q_{\pm}^n$. 
(Possibly $T_+^n=T_-^n$.)
Since the injectively immersed flat triangles $T_{\pm}^n$ 
meet $\De_{\Si}$
only along their hypotenuses $\De_{s_n}$, 
we can lift them to isometrically embedded flat triangles 
$\widetilde T_{\pm}^n$ in $X$ with hypotenuses $\widetilde s_n$, 
as we could lift the half-planes in ${\mathcal H}$
to half-planes in $\widetilde{\mathcal H}$. 
The lifts are again uniquely determined by the lift 
of one off-hypotenuse point. 
Thus we can choose them such that 
$\widetilde q_{\pm}^n\in\widetilde T_{\pm}^n\subset\pi^{-1}(c_n\times c_n)$. 
Then the pointed triangles $(\widetilde T_{\pm}^n,\widetilde q_{\pm}^n)$
Hausdorff converge to a flat half-plane 
in $\widetilde{\mathcal H}$ with base point $\widetilde q_{\pm}$
and boundary line $\widetilde l$. 
The only such half-plane is $\widetilde H_{\pm}$,
i.e.\ $(\widetilde T_{\pm}^n,\widetilde q_{\pm}^n)
\to(\widetilde H_{\pm},\widetilde q_{\pm})$. 

The two triangles $T_{\pm}^n$ either coincide or have angle $\pi$ 
along their common side $\De_{s_n}$. 
The local geometry of branched coverings implies 
that the lifted triangles $\widetilde T_{\pm}^n$ 
have angle $\pi$ along their common side $\widetilde s_n$. 
(They are distinct for large $n$, 
$\widetilde T_+^n\cap\widetilde T_-^n=\widetilde s_n$.)
Hence their union 
$\widetilde Q^n=\widetilde T_+^n\cup\widetilde T_-^n$
is an embedded flat square in $X$. 
These are the squares we were looking for. 
As desired, 
they satisfy 
$(\widetilde Q^n,\widetilde p_n)\to(F,\widetilde p)$. 
This finishes the proof. 
\end{proof}

%\subsection{Non-smoothability}

Now we describe a configuration in $X$ 
which satisfies conditions (i-vii) 
formulated in sections~\ref{sec:confnotoccsm} and~\ref{sec:nequivsm}. 

We consider a decomposition (\ref{eq:deco}) of $\Si$ 
and choose an injectively immersed geodesic line $c\subset\Si$ 
which intersects $\Si^+\cap\Si^-$ transversally 
in precisely one point $p$. 
The geodesic $c$ is divided by $p$ into the injectively immersed rays 
$r^{\pm}=c\cap\Si^{\pm}$. 
We can arrange our choices 
(of $\Si$, $\Si^{\pm}$ and $c$) 
so that 

(a)
$r^{\pm}$ is asymptotic to a simple closed geodesic 
$c^{\pm}\subset\interior(\Si^{\pm})$, and 

(b)
$c$ is a pointed Hausdorff limit 
of simple closed geodesics $c_n\subset\Si$.

Indeed, 
if $\Si^{\pm}$ and $c^{\pm}$ are chosen appropriately 
then there exists a simple closed curve $a$,  
which intersects $c^+$ and $c^-$ transversally
in one point each 
and $\Si^+\cap\Si^-$ transversally in two points. 
It is divided by its intersection points with $c^{\pm}$ 
into two arcs $a_{+-}$ and $a_{-+}$. 
The concatenations $a_{+-}*nc^-*a_{-+}*nc^+$ 
are freely homotopic to simple closed geodesics $c_n$ which, 
when equipped with suitable base points, 
Hausdorff converge to an injectively immersed line $c$ 
with the desired properties.  

Let $H\in{\mathcal H}$ be the half-plane $H\subset c\times c$ 
with boundary line $\D H=\De_c$ and containing the quadrant $r^+\times r^-$. 
There exist two distinct flats $F_1,F_2\in{\mathcal F}_1$ 
which contain the same lift 
$\widetilde H\in\widetilde{\mathcal H}$ of $H$
(and branch along its boundary line $\D\widetilde H$). 
Their union $F_1\cup F_2$ splits metrically as $Z\times\R$,
and the cross section $Z$ is a degenerate ideal triangle (a tripod). 
By Lemma~\ref{lem:limit}, 
the three flats contained in $Z\times\R$,
i.e.\ $F_1,F_2$ and $(F_1\cup F_2)-\interior(\widetilde H)$, 
are $\Ga$-periodically approximable. 

Let $\widetilde r^+\times\widetilde r^-\subset\widetilde H$ 
be the quadrant lifting $r^+\times r^-$. 
There exists a closed convex product subset $P=Y^+\times Y^-\subset X$ 
such that $\pi|_P$ is a universal covering of $\Si^+\times\Si^-$ 
and $F_j\cap P=\widetilde r^+\times\widetilde r^-$ for $j=1,2$. 

The product subsets $Y^+\times Y^-$ and $Z\times\R$ 
satisfy conditions (i)-(vii). 
Applying Theorem~\ref{thm:obst}, 
we therefore obtain: 
\begin{reptheorem}{mainthm}[Exercise 1 in \cite{BGS}]
Let $V$ be a closed 4-dimensional manifold which admits a non-trivial finite branched covering $\beta:V\rightarrow\Sigma\times\Sigma$ 
over the product of a hyperbolic surface $\Sigma$ with itself 
such that the 
branching locus equals the diagonal $\Delta_\Sigma\subset\Sigma\times\Sigma$. 
Then $V$ admits no smooth Riemannian metric
of nonpositive sectional curvature. 
\end{reptheorem}
\begin{remark}
As in Theorem~\ref{thm:obst}, 
one can relax the regularity assumptions 
and rule out 
the existence of $\mathcal{C}^2$-smooth Riemannian metrics on $V$
\cite{thesis}. 
\end{remark}

\end{document}